\documentclass[12pt]{article}
\usepackage{color}
\usepackage{amsmath}
\usepackage{amsfonts,amssymb}
\usepackage{extarrows}
\usepackage{geometry}
\usepackage{authblk} 
\usepackage{hyperref} 
\usepackage{mathrsfs}
\usepackage{multirow}
\usepackage{diagbox}
\usepackage{threeparttable}
\usepackage[utf8]{inputenc}
\def\0{\emptyset}

\newtheorem{theorem}{Theorem}[section]

\newtheorem{lemma}[theorem]{Lemma}
\newtheorem{claim}[theorem]{Claim}

\newenvironment{proof}{{\noindent\it Proof.}}{\hfill $\square$\par}

\newcommand{\ex}{\mathrm{ex}}
\newcommand{\ind}{\mathrm{ind}}

\usepackage{graphicx}

\usepackage{mathtools} 

\usepackage{enumerate}
\usepackage{enumitem}

\usepackage{
	amsmath,			
	amssymb,			
	enumerate,		    
	graphicx,			
	lastpage,			
	multicol,			
	multirow,			
	pifont,			    
}

\usepackage[numbers]{natbib}

\begin{document}


\title{The number of induced paths in outerplanar graphs}

\author[1]{\small\bf Yichen Wang\thanks{E-mail: wangyich22@mails.tsinghua.edu.cn}}
\author[2]{\small\bf Ervin Gy\H{o}ri\thanks{E-mail:  gyori.ervin@renyi.hu}}
\author[2]{\small\bf Casey Tompkins\thanks{E-mail: tompkins.casey@renyi.hu}}
\author[1]{\small\bf Xiamiao Zhao\thanks{Corresponding author: E-mail: zxm23@mails.tsinghua.edu.cn}}

\affil[1]{\small Department of Mathematical Sciences, Tsinghua University, Beijing, P.R. China.}
\affil[2]{\small HUN-REN Alfr\'ed R\'enyi Institute of Mathematics, Budapest, Hungary.}


\date{}

\maketitle\baselineskip 16.3pt

\begin{abstract}
    Let $P_k$ denote the path with $k$ vertices, and $\mathrm{ex}_{\mathcal{OP}}(n,H^{\mathrm{ind}},\emptyset)$ be the maximum number of induced copies of $H$ in an $n$-vertex outerplanar graph.
    In this paper, we determine the exact value of $\mathrm{ex}_{\mathcal{OP}}(n,P_3^{\mathrm{ind}},\emptyset)$ for all $n$, and give an asymptotic value of $\mathrm{ex}_{\mathcal{OP}}(n,P_4^{\mathrm{ind}},\emptyset)$.
    For general $k$, Matolcsi and Nagy proved that $\lim_{k\to \infty} {\left( \ex_{\mathcal{OP}}(n, P_{k+1},\emptyset)\right)^{1/k}} =4$.
    In the induced case, we prove that 
    \[
        fib(k-1)\frac{{(n-2k+3)}^2}{4} \le \ex_{\mathcal{OP}}(n, P_{k+1}^{\mathrm{ind}},\emptyset) \le fib(k+1) \binom{n}{2},
    \]
    where $fib(k)$ is the Fibonacci number. 
    This implies that $\lim_{k\to \infty} {\left( \ex_{\mathcal{OP}}(n, P_{k+1}^{\mathrm{ind}},\emptyset)\right)^{1/k}} = \frac{\sqrt{5}+1}{2}\approx 1.618$.
    


    
\end{abstract}




\section{Introduction}\label{sec: intro}
This paper is about the Tur\'an problem on planar graphs. 
More specifically, we focus on a special class of planar graphs, say outerplanar graphs.
An outerplanar graph is a graph that has a planar drawing for which all vertices belong to the outer face of the drawing.
If the outer face is a cycle, we call it the outer cycle.
In this paper, let $P_k$ denote the path with $k$ vertices, and $C_k$ denote the cycle with $k$ vertices.
Chartrand and Harary~\cite{chartrand1967planar} introduced the concept of outerplanar graphs, and they proved that a graph is outerplanar if and only if it does not contain $K_4$ or $K_{2,3}$ as a minor.

The generalized Tur\'an problem is to determine the maximum number of copies of a graph $H$ in an $n$-vertex graph that does not contain any graph $F\in \mathcal{F}$ as a subgraph, which is denoted by $\ex(n,H,\mathcal{F})$, where $\mathcal{F}$ is a family of graphs.
If we restrict the host graph to be planar, then we get the planar Tur\'an number $\ex_{\mathcal{P}}(n,H,\mathcal{F})$.
And if we restrict the host graph to be outerplanar, then we get the outerplanar Tur\'an number $\ex_{\mathcal{OP}}(n,H,\mathcal{F})$.

The generalized planar Tur\'an problems were systematically investigated by Gy\H{o}ri, Paulos, Salia, Tompkins and Zamora~\cite{gyHori2020generalized}.
Ghosh, Gy\H{o}ri, Martin, Paulos, Salia, Xiao and Zamora~\cite{GHOSH2021112317} determined $\ex_{\mathcal{P}}(n, P_5, \emptyset)$.
Grzesik, Gy\H{o}ri, Paulos, Salia, Tompkins and Zamora~\cite{Grzesik2022} determined $\ex_{\mathcal{P}}(n, P_4, \emptyset)$.
Lv, Gy\H{o}ri, He, Salia, Tompkins and Zhu~\cite{LV202415} determined $\ex_{\mathcal{P}}(n, C_{2k}, \emptyset)$ for all $k \ge 2$ asymptotically.
Gy\H{o}ri, Paulos, Salia, Tompkins and Zamora~\cite{gyHori2019maximum} determined $\ex_{\mathcal{P}}(n, C_5, \emptyset)$.
Gy\H{o}ri and Hama Karim~\cite{gyHori2024generalized} determined $\ex_{\mathcal{P}}(n, C_\ell, C_3)$ for $\ell\in\{4,5,6\}$ and $\ex_{\mathcal{P}}(n, C_3, C_\ell)$ for $\ell\in\{4,5,6\}$.

There are many results when counting the maximum number of induced copies of $H$ in an $n$-vertex, $\mathcal{F}$-free planar graph, which is denoted by $\ex(n,H^{\ind},\mathcal{F})$.
Savery~\cite{savery2024planar} determined the exact value of $\ex_{\mathcal{P}}(n, C_4^{\ind}, \emptyset)$.
Independently, Ghosh, Gy\H{o}ri, Janzer, Paulos, Salia and Zamora~\cite{ghosh2022maximum} also determined $\ex_{\mathcal{P}}(n, C_5^{\ind}, \emptyset)$.
Later, Savery~\cite{savery2021planar} determined $\ex_{\mathcal{P}}(n, C_6^{\ind}, \emptyset)$ too.

When the host graph is outerplanar, there are more results about it.
Eppstein~\cite{eppstein1993connectivity} proved that $\ex_{\mathcal{OP}}(n,H, \emptyset)=O(n)$
if and only if $H$ is $2$-connected.
For a tree $T$, let $\ell(T)$ denote the number of leaves in $T$, i.e., the number of vertices $v\in V(T)$ of degree one.
Huynh, Joret and Wood~\cite{huynh2022subgraph} proved that for any tree $T$, $\ex_{\mathcal{OP}}(n,T,\emptyset)=\Theta(n^{\ell(T)})$.
Matolcsi and Nagy~\cite{MATOLCSI2022115} proved that there is a constant $c(k)$ that depends only on $k$ such that $\ex_{\mathcal{OP}}(n, C_{k}, \emptyset) = c(k) n + O(1)$.
Matolcsi and Nagy~\cite{MATOLCSI2022115} also showed that there is a constant $g(k)$ depending on $k$ only, such that $g(k) \binom{n}{2} < \ex_{\mathcal{OP}}(n, P_{k+1}, \emptyset) \le 4^k \binom{n}{2}$, and $\lim_{k\to \infty}g(k)^{\frac{1}{k}}=4$.
For small paths,
Gy\H{o}ri, Paulos, Xiao and Zamora~\cite{GYORI2023113205} determined $\ex_{\mathcal{OP}}(n, P_4, \emptyset) = 2n^2-7n+2$ and $\ex_{\mathcal{OP}}(n, P_5, \emptyset) = \frac{17}{4}n^2 +\Theta(n)$.

Motivated by the study of the generalized Tur\'an number, we study the value of \\ $\ex_{\mathcal{OP}}(n,P_k^{\ind},\emptyset)$.
First, we have an exact result when $k=3$.

\begin{theorem}\label{thm: P3 inducibility}
    For every $n\geq 7$, we have
    $$\ex_{\mathcal{OP}}(n,P_3^{\ind},\emptyset)= \binom{n-1}{2}.$$
    Moreover, the equality holds only when $G$ is a star.
    For $n=4$, the extremal graph is the cycle on 4 vertices, which contains $4$ induced $P_3$'s.
    For $n=5$, the extremal graph is the cycle with one pendant leaf, which contains $6$ induced $P_3$'s.
    For $n=6$, the extremal graph is the star or a $C_6$ with one chord joining opposite vertices which contains $10$ induced $P_3$'s.
\end{theorem}

When $k=4$, we provide an asymptotic result.
\begin{theorem}\label{thm: P4 inducibility}
    There exists a constant $C$ such that
    the number of induced $P_4$'s in every outerplanar graph on $n$ vertices is at most $\frac{1}{2}n^2+C n\log n$.
    The lower bound can be achieved asymptotically by the construction in~Theorem~\ref{thm: Pk outerplanar}.
\end{theorem}

Let $fib(t)$ be the Fibonacci number with $fib(1)=fib(2)=1$ and $fib(t)=fib(t-1)+fib(t-2)$ for $t\geq 3$.
Then for general $k$, we have a similar form of result as in~\cite{MATOLCSI2022115}.
\begin{theorem}\label{thm: Pk outerplanar}
    When $n$ is sufficiently large, there is a constant $h(k)$ determined by $k$ such that
    \begin{equation*}
    \begin{aligned}
        fib(k-1)\frac{{(n-2k+3)}^2}{4} \le \ex_{\mathcal{OP}}(n, P_{k+1}^{\ind},\emptyset) \le fib(k+1) \binom{n}{2},
    \end{aligned}
    \end{equation*}
\end{theorem}

It is proved in~\cite{MATOLCSI2022115} that $\lim_{k\to \infty} {\left( \ex_{\mathcal{OP}}(n, P_{k+1},\emptyset)\right)^{1/k}} =4$.
Then in the non-induced version, Theorem~\ref{thm: Pk outerplanar} implies that $\lim_{k\to \infty} {\left( \ex_{\mathcal{OP}}(n, P_{k+1}^{\ind},\emptyset)\right)^{1/k}} = \frac{\sqrt{5}+1}{2}\approx 1.618$.
 
In Section 2, we give some lemmas that are useful in our proof.
In Section 3, we give the proof of Theorem~\ref{thm: P3 inducibility} and Theorem~\ref{thm: P4 inducibility}. 
In Section 4, we prove Theorem~\ref{thm: Pk outerplanar}.
We also consider the planar version of this problem, which will be discussed in an upcoming paper.

\section{Preliminaries}\label{sec:Preliminaries}
For a graph $G$, $v(G)$ and $e(G)$ denote the number of vertices and the number of edges, respectively.
For a vertex set $U$ of $G$, $G[U]$ denotes the subgraph of $G$ induced by $U$.
We use $N_G(v)$ to denote the set of neighbors of $v$ in the graph $G$.

\begin{lemma}[\cite{GYORI2023113205}]\label{lem: ervin-tree edge cut}
    Let $T$ be an $n$-vertex tree with $\Delta(T)\leq 3$. Then there is an edge $e\in E(T)$ such that both components of $T-e$ have at least $\frac{n-1}{3}$ vertices.
\end{lemma}
We generalize this lemma to all the $k$ with $\Delta(T)\leq k$.
\begin{lemma}\label{lem: tree edge cut}
    Let $T$ be an $n$-vertex tree with $\Delta(T)\leq k$ for $k\geq 3$. Then there is an edge $e\in E(T)$ such that both components of $T-e$ have at least $\frac{n-1}{k}$ vertices.
\end{lemma}
\begin{proof}
    We prove by induction on $k$. 
    When $k=3$, then it holds by Lemma \ref{lem: ervin-tree edge cut}.
    For $k\geq 4$, let $T$ be a tree with $\Delta(T)\leq k$. We may assume the number of vertices with degree $k$ is at least one.
    Then, suppose $v$ has degree $k$ with neighbors $\{v_1,\dots,v_k\}$. Let $T_i$ denote the maximal sub-tree of $T$ containing the vertex $v_i$ but not $v$. We let $v(T_i)=n_i$.
    And we prove by the number of $k$-degree vertices of $T$, which is denoted by $t$ and $t\geq 1$.

    Clearly, at least one of $n_i$ is at least $\frac{n-1}{k}$, say $n_1\geq \frac{n-1}{k}$. 
    And if $\sum_{i=2}^k n_i\geq \frac{n-1}{k}$, we are done by choosing $e=vv_1$.
    Then, assume $\sum_{i=2}^k n_i <\frac{n-1}{k}$. And we consider the tree $T'$ obtained from $T$ by deleting the vertices in $\bigcup_{i=2}^k T_i$ and identifying the vertex $v$ and a terminal vertex (say $r_1$) of a $(\sum_{i=2}^k n_i+1)$-vertex path $(r_1,r_2,\dots, r_{\sum_{i=2}^k n_i+1})$.
    The new tree $T'$ has at most $t-1$ vertices with degree $k$, and with $n$ vertices.
    Then, by the induction hypothesis, there is an edge $e$ such that both components of $T'-e$ have at least $\frac{n-1}{k}$.
    Moreover, $e$ is not in the path $(v,r_1,\dots,r_{\sum_{i=2}^k n_i+1})$. Therefore, with the choice of $e$, both components of $T-e$ have at least $\frac{n-1}{k}$ vertices.
\end{proof}
\section{Count induced \texorpdfstring{$P_3$}{P3} and \texorpdfstring{$P_4$}{P4} in outerplanar graphs}\label{sec: induce P4 outerplanar}

As a warm-up, we first determine the value of $\ex_{\mathcal{OP}}(n,P_3^{\ind},\emptyset)$. 
Similar ideas in this proof will be used in the proof of Theorem~\ref{thm: P4 inducibility}.

\noindent
\textit{Proof of Theorem \ref{thm: P3 inducibility}.}
    The case when $n \le 6$ can be checked easily.
    Now let us prove the case when $n \ge 7$.
    The lower bound is achieved by a star on $n$ vertices.
    Then, we focus on the upper bound.
    We will prove it by induction on $n$. 
    When $n = 7$, one can check by a computer program that the upper bound holds and the unique extremal graph is the star on 7 vertices.
    For $n\geq 8$, let $G$ be the $n$-vertex outerplanar graph with the maximum number of induced $P_3$'s.

    If $G$ is not $2$-connected, then suppose $v$ is a cut vertex.
    There exists vertex set $U, U'$ satisfying $V(G) = U\cup U'$, $U\cap U'=\{v\}$, and all edges are in $U$ or $U'$.
    Let $|U|=n_1$, $|U'|=n_2$, then $n_1+n_2=n+1$.
    Let $H = G[U], H' = G[U']$.
    Then, by induction, the number of induced $P_3$'s contained in $H$ is at most $\binom{n_1-1}{2}$, and the number of induced $P_3$'s in $H'$ is at most $\binom{n_2-1}{2}$, and the equalities hold only when $H$ and $H'$ are both stars.
    The number of induced $P_3$'s containing vertices from both $U\setminus \{v\}$ and $U'\setminus \{v\}$ is at most $(n_1-1)(n_2-1)$.
    As a result, the number of induced $P_3$'s contained in $G$ is at most
    $$\binom{n_1-1}{2}+\binom{n_2-1}{2}+(n_1-1)(n_2-1)\leq \binom{n-1}{2}.$$
    And the equality holds only when $G$ is a star with $v$ as the center vertex.

    If $G$ is $2$-connected, then we draw $G$ in the plane.
    By~\cite{leydold1998minimal}, $G$ has a unique Hamiltonian cycle, which is the outer cycle.
    Suppose there is no chord in this cycle, which implies $G$ is a cycle with $n$ vertices, then the number of induced $P_3$'s is at most $n<\binom{n-1}{2}$ when $n\geq 6$.
    And suppose $uv$ is a chord in the outer cycle, it partitions the outer cycle into two parts $U$ and $U'$ where $U\cap U'=\{u,v\}$.
    We let $|U|=n_1$, $|U'|=n_2$, $H = G[U]$, $H' = G[U']$, then $n_1+n_2=n+2$.
    And notice that neither $H$ nor $H'$ is a star, the number of induced $P_3$'s contained in each $H,H'$ is at most $\binom{n_i-1}{2}-1$ by induction.
    Let $x = |N_{H}(u) \setminus \{v\}|, x' = |N_{H'}(u) \setminus \{v\}|, y = |N_{H}(v) \setminus \{u\}|, y' = |N_{H'}(v) \setminus \{u\}|$, then we have $x,y \le n_1-2$, $x',y' \le n_2-2$.
    Then the number of induced $P_3$'s not contained in $H$ or $H'$ is at most $xx'+yy'$.
    By the outerplanarity, $u$ and $v$ have at most one common neighbor in $H$ and $H'$ respectively, then $x+y\leq n_1-1$, $x'+y'\leq n_2-1$.
    Then $xx'+yy'\leq (n_1-2)(n_2-2)+1$.
    As a result, the number of induced $P_3$'s contained in $G$ is at most
    $$\binom{n_1-1}{2}+\binom{n_2-1}{2}-2+(n_1-2)(n_2-2) + 1 < \binom{n-1}{2}.$$
    \hfill $\square$ \par
    
\noindent
Next, we give the proof of Theorem \ref{thm: P4 inducibility}.

\noindent
\textit{Proof of Theorem \ref{thm: P4 inducibility}}
    The lower bound will be defined in Section~\ref{sec: induce Pk outerplanar} and can be verified easily.
    We omit the proof of the lower bound here.
    Now we prove the upper bound.
    Let $f(n) = \frac{1}{2}n^2 + Cn\log n$, where $C$ is a constant.
    We will prove it by induction on $n$.
    Let $N_0 \ge 400$ be a large constant to be determined later.
    When $n\leq N_0$, by picking $C$ large enough, the base case holds.

    Suppose the theorem holds for all outerplanar graphs on at most $n-1$ vertices, then now we prove the case for $n$.
    Let $G$ be an $n$-vertex outerplanar graph with maximum number of induced $P_4$'s, and for the same number of induced $P_4$'s, we pick $G$ with the maximum number of edges.
    Let $G'\supseteq G$ be a maximal outerplanar graph containing $G$ as a subgraph. 
    Then $G'$ admits a planar drawing with an outer cycle $\mathcal{C}$.
    We get a drawing of $G$ from this embedding by removing the edges not in $G$.
    Let $G''=G+\mathcal{C}$ be the graph obtained by adding edges in $\mathcal{C}$ to $G$.
    In the following, we fix the drawing of $G$.
    Note that all edges of $G$ are inside $\mathcal{C}$ except the edges in $\mathcal{C}$.
    \begin{claim}\label{claim: xy}
        When $n\geq 400$,
        there exists an edge $xy\in E(G)$ such that $xy$ divides $\mathcal{C}$ into two graphs $U$ and $U'$ with $n_1$ and $n_2$ vertices respectively, such that $V(U)\cap V(U')=\{x,y\}$ and  $n_1,n_2 \ge \frac{n}{200}$.
    \end{claim}
    \noindent
    \textit{Proof of Claim~\ref{claim: xy}:}
        In the graph $G''$, we construct an auxiliary graph $H$~(the dual graph of $G''$) as follows.
        The vertex set of $H$ is the bounded faces of $G''$, and we add an edge between two faces if they share an edge in $G''$.
        Then, since $G''$ is outerplanar, $H$ must be a tree.
        Note that each edge in $G$ except the edges in $\mathcal{C}$ corresponds to an edge in $H$.
        We may assume $H$ is not empty, otherwise, $G$ is a subgraph of $\mathcal{C}$, then $G$ contains at most $n \le f(n)$ induced $P_4$, a contradiction.

        We claim that each face of $G''$ has at most $10$ edges. 
        Suppose there exists a face $f$ with at least $11$ edges.
        We label the vertices of $f$ as $u_{1}, u_2, \ldots,u_{t}$ in clockwise order~($t \ge 11$).
        By outerplanarity, the distance between $u_1$ and $u_6$ in $G$ is at least five.
        Then if we add the edge $u_1u_6$ to $G$, the number of induced $P_4$'s would not decrease while the graph has more edges, a contradiction.
        This implies that $\Delta(H)\leq 10$.

        
        Then by Lemma \ref{lem: tree edge cut}, there exists an edge $e$ in $H$ such that both components of $H-e$ have at least $\frac{n-1}{10}$ vertices.
        Let $xy$ be the edge in $G$ corresponding to $e$, then $xy$ divides $\mathcal{C}$ into two parts $U$ and $U'$ with $n_1$ and $n_2$ vertices respectively, such that $V(U)\cap V(U')=\{x,y\}$.
        Moreover, let $f$ be the number of bounded faces in $G''$, both $G''[U]$ and $G''[U']$ have at least $\frac{f-1}{10}$ faces.

        Since each bounded face in $G''$ has at most $10$ vertices and at least three vertices, combining Euler's formula, we have $f \le n \le 10f$.
        Then the number of vertices in $U$ is at least $f_1 \ge \frac{f-1}{10} \ge \frac{n}{200}$ where $f_1$ is the number of bounded faces in $G''[U]$.
        The same conclusion holds for vertices in $U'$.
        \hfill $\blacksquare$ \par

    \begin{figure}
        \centering
        \includegraphics[width=0.4\linewidth]{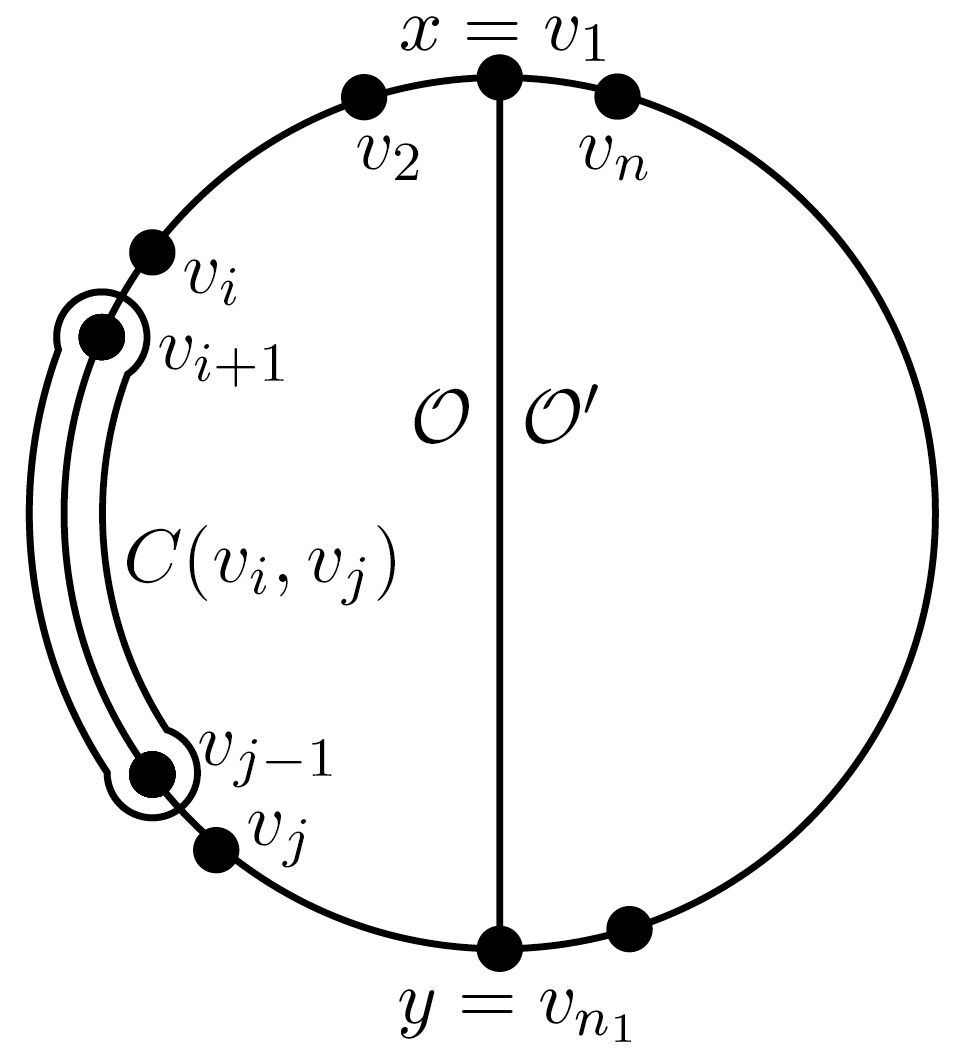}
        \caption{An example for $\mathcal{O}, \mathcal{O}'$ and $C(v_i,v_j)$.}
        \label{fig:def cuv}
    \end{figure}

    From now on, we fix an edge $xy$ as in Claim~\ref{claim: xy}.
    Let $U$ and $U'$ be the two parts of $\mathcal{C}$ divided by $xy$ with $n_1$ and $n_2$ vertices respectively, such that $V(U)\cap V(U')=\{x,y\}$,$n_1, n_2 \ge \frac{n}{200}$ and $n_1+n_2=n+2$.
    Let $\mathcal{O}=G[U]$ and $\mathcal{O}'=G[U']$.

We may assume $x=v_1,v_2,\dots,v_{n_1}=y,v_{n_1+1},\dots,v_{n}$ are the vertices in $\mathcal{C}$ in counter-clockwise order, and $V(\mathcal{O})=\{v_1,\dots,v_{n_1}\}$, $V(\mathcal{O}')=\{v_{n_1+1},\dots,v_{n_1+n_2}\}$.
For two vertices $v_i,v_j \in U\setminus \{x,y\}, i<j$, let $C(v_i,v_j) = \{v_{i+1}, \ldots, v_{j-1}\}$~(see Figure~\ref{fig:def cuv}).

Let $\mathcal{O}(x) = N_{\mathcal{O}}(x)\setminus \{y\}$, $\mathcal{O}(y) = N_{\mathcal{O}}(y)\setminus \{x\}$.
Similarly, we define $\mathcal{O}'(x)$ and $\mathcal{O}'(y)$.
Let $\phi(xy)$ be the number of induced $P_4$'s that intersect both $U\setminus\{x,y\}$ and $U'\setminus\{x,y\}$, i.e., the induced $P_4$'s not contained in either $\mathcal{O}$ or $\mathcal{O}'$.
Then it is obvious that all such induced $P_4$ must contain at least one of $\{x,y\}$.

Let $s_1 = \left|\mathcal{O}(x) \setminus \{y\}\right|$ be the number of neighbors of $x$ in $\mathcal{O}$ except $y$ and $s_2$ be the number of induced $P_3$'s starting at $x$ and avoiding $y$ in $\mathcal{O}$.
Similarly, we define $t_1,t_2$ for $x$ in $\mathcal{O}'$, $p_1,p_2$ for $y$ in $\mathcal{O}$, $q_1,q_2$ for $y$ in $\mathcal{O}'$.
Then there are six types of induced $P_4$ counted in $\phi(xy)$, which are shown in Figure~\ref{fig: six types}.
Then we have
\begin{equation}\label{eq: phi}
    \phi(xy)\leq s_1q_1+t_1p_1+s_1t_2+s_2t_1+p_1q_2+p_2q_1.
\end{equation}

\begin{figure}
    \centering
    \includegraphics[width=\linewidth]{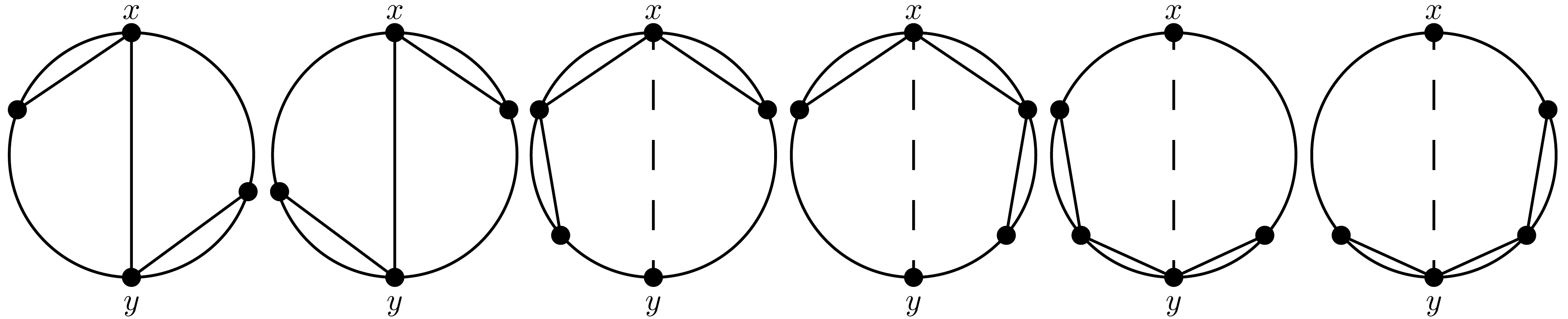}
    \caption{Different types of induced $P_4$ in $\phi(xy)$.}
    \label{fig: six types}
\end{figure}

Take a vertex $v \in U\setminus \{x,y\}$.
If $v \in N_{\mathcal{O}}(x) $, we say that $v$ contributes one to $s_1$ and $v$ contributes zero to $s_2$.
If $v \notin N_{\mathcal{O}}(x)$, we say the \textit{contribution} of $v$ to $s_2$ is the number of induced $P_3$'s starting at $x$, avoiding $y$ in $\mathcal{O}$, and ending at $v$.
In this case, $v$ contributes zero to $s_1$.
Then we can estimate $s_1$ and $s_2$ by the sum of the contributions of all vertices in $U\setminus \{x,y\}$. 
By outerplanarity, we have the following facts.

\begin{itemize}
    \item \textbf{Fact 1.} Suppose $v_i\in \mathcal{O}(x)$, then for every $j<i$, $v_j \notin \mathcal{O}(y)$. 
    Similarly, suppose $v_i \in \mathcal{O}(y)$, then for every $j > i$, $v_j \notin \mathcal{O}(x)$. 
    Then as a result, we may assume $\mathcal{O}(x) = \{v_{i_1},v_{i_2},\dots,v_{i_{s_1}}\}$ and $\mathcal{O}(y) = \{v_{j_1},v_{j_2},\dots,v_{j_{p_1}}\}$ in counter-clockwise order and $j_1 \ge i_{s_1}$.
    \item \textbf{Fact 2.} If $i_{s_1} < j_1$~($v_\ell$ does not exist), then for every vertex in $C(v_{i_{s_1}},v_{j_1})$, it contributes at most one to $s_2$ and at most one to $p_2$.
    \item \textbf{Fact 3.} For every two consecutive vertices $v_{i_m},v_{i_{m+1}} \in \mathcal{O}(x)$, at most one vertex in $C(v_{i_m},v_{i_{m+1}})$ contributes two to $s_2$.
\end{itemize}


Then, we partition the vertices in $U \setminus \{x,y\}$ into several parts, and calculate the contribution of each part to $s_1,s_2,p_1,p_2$.

\begin{itemize}
    \item Let $A$ be the collection of all the vertices not in $\mathcal{O}(x)\cup \mathcal{O}(y)$ that have at most one neighbor in $\mathcal{O}(x)$ and at most one neighbor in $\mathcal{O}(y)$.
    \item Let $B_1$ be the collection of vertices $v_i \in \mathcal{O}(x)$ such that there exists a vertex in $\mathcal{O}(x)$ after $v_i$, say $v_j$~($j > i$ and $C(v_i,v_j) \cap \mathcal{O}(x) = \emptyset$), and $v_i,v_j$ have no common neighbor in $C(v_i,v_j)$.
    \item Similarly, we define $B_2$ for $y$ but in a reverse order. 
    Let $B_2$ be the collection of vertices $v_i \in \mathcal{O}(y)$ such that there exists a vertex in $\mathcal{O}(y)$ before $v_i$, say $v_j$~($j < i$ and $C(v_j,v_i) \cap \mathcal{O}(y) = \emptyset$), and $v_j,v_i$ have no common neighbor in $C(v_j,v_i)$.
    \item Let $D_1=C(v_{i_1},v_{i_{s_1}})\cup \{v_{i_1},v_{i_{s_1}}\}\setminus(A\cup B_1).$
    \item Let $D_2=C(v_{j_1},v_{j_{p_1}})\cup \{v_{j_1},v_{j_{p_1}}\}\setminus(A \cup B_2).$
\end{itemize}

By Facts 1, 2 and 3, every vertex in $U\setminus \{x,y\}$ belongs to exactly one of $A,B_1,B_2,D_1,D_2$ except when $i_{s_1} = j_{1}$, the vertex $v_{j_1}$ belongs to both $D_1$ and $D_2$.
When $i_{s_1} = j_{1}$, we let $\ell = i_{s_1} = j_{1}$ and we say $v_\ell$ exists in this case.
We use corresponding lower case letter to denote the size of the set.
Then we have $a +b_1+b_2+d_1+d_2\le n_1-1$.
Next, we calculate the contribution of each part to $s_1,p_1$ and $s_2,p_2$~(see Table~\ref{tab: contribution} for details).

\begin{enumerate}
    \item There exists at most one vertex $v_x \in B_1 \cup D_1$ adjacent to $v_{j_1}$. In this case, $v_x$ contributes one to $p_2$. Similarly, there exists at most one vertex $v_y \in B_2 \cup D_2$ adjacent to $v_{i_{s_1}}$. In this case, $v_y$ contributes one to $s_2$.
    \item For every vertex in $A$, it contributes zero to $s_1,t_1$, at most one to $s_2$ and at most one to $p_2$.
    \item For every vertex in $B_1$, it contributes one to $s_1$ and zero to $s_2,p_1$. $B_1$ contributes at most one to $p_2$ only when $v_x \in B_1$.
    Symmetrically, for every vertex in $B_2$, it contributes one to $p_1$ and zero to $s_1,p_2$. And $B_2$ contributes at most one to $s_2$ only when $v_y \in B_2$.
    \item We consider the contribution of vertices in $D_1$ together. By Facts 1, 2 and 3 and the definition of each set, one can verify that $D_1$ must have an odd number of vertices and the vertices are alternately contained in $\mathcal{O}(x)$, that is, exactly $\frac{1}{2}(|D_1|+1)$ vertices are in $\mathcal{O}(x)$ which contributes $\frac{1}{2}(|D_1|+1)$ to $s_1$, and at most one to $p_1$ only when $v_\ell$ exists and contributes by $v_{\ell}$.
    For the rest of $\frac{1}{2}(|D_1|-1)$ vertices in $D_1$, each of them contributes at most two to $s_2$, and at most one to $p_2$ which only happens when $v_x \in D_1$.
    \item Similarly we have the contribution of vertices in $D_2$.
\end{enumerate}



\begin{table}[t]\label{tab: contribution}
\centering
\begin{tabular}{|c|c|c|c|c|}
\hline
\diagbox{Sets}{Terms}&$s_1$&$s_2$&$p_1$&$p_2$\\ 
\hline
$A$&~&$\leq|A|$&~&$\leq|A|$\\
\hline
$B_1$&$|B_1|$&~&~&\text{$1$ when $v_x \in B_1$}\\
\hline
$B_2$&~&\text{$1$ when $v_y \in B_2$}&$|B_2|$&~\\
\hline
$D_1$&$\frac{1}{2}(|D_1|+1)$&$|D_1|-1$& \text{$1$ when $v_\ell$ exists}&\text{$1$ when $v_x \in D_1$}\\
\hline
$D_2$&\text{$1$ when $v_\ell$ exists}& \text{$1$ when $v_y \in D_2$} & $\frac{1}{2}(|D_2|+1)$&$|D_2|-1$\\
\hline
\text{$v_\ell$ when exists}&1&~&1&~\\
\hline
\end{tabular}
\caption{The contribution of each set to $s_1,s_2,p_1,p_2$.}
\end{table}


 
Note that when $v_\ell$ exists, the contribution of $v_\ell$ to $s_1$ and $p_1$ is counted twice in both $D_1$ and $D_2$.
Then, by summing up all the contributions, we have
\begin{equation}\label{eq: inequalities 1}
    \begin{aligned}
        & a + b_1 + b_2 + d_1 + d_2 \leq n_1-1, \\
        &s_1 \leq \frac{1}{2}(d_1+1)+b_1,\\
        &p_1 \leq \frac{1}{2}(d_2+1)+b_2,\\
        &s_2 \leq d_1-1+a+1,\\
        &p_2 \leq d_2-1+a+1.
    \end{aligned}
\end{equation}
Symmetrically, we can define $A',B_1',B_2',D_1',D_2'$ and $v_\ell'$ in $\mathcal{O}'$, and then we have 
\begin{equation}\label{eq: inequalities 2}
    \begin{aligned}
        & a' + b_1' + b_2' + d_1' + d_2' \leq n_2-1, \\
        &t_1 \leq \frac{1}{2}(d_1'+1)+b_1',\\
        &q_1 \leq \frac{1}{2}(d_2'+1)+b_2',\\
        &t_2 \leq d_1'-1+a'+1,\\
        &q_2 \leq d_2'-1+a'+1.
    \end{aligned}
\end{equation}

Then we claim that by combining (\ref{eq: phi}), (\ref{eq: inequalities 1}) and (\ref{eq: inequalities 2}), we have the following upper bound for $\phi(xy)$:
\begin{equation}\label{eq: phi upper bound}
\begin{aligned}
    \phi(xy)&\leq n_1n_2+n_1+n_2.
\end{aligned}
\end{equation}
Let $a_0 = a + 1$, $a_0' = a+1$ for convenience, then we have
\begin{align*}
    \phi(xy)&\leq \left(\frac{1}{2}(d_1+1)+b_1\right)\left(\frac{1}{2}(d_2'+1)+b_2'\right)+\left(\frac{1}{2}(d_1'+1)+b_1'\right)\left(\frac{1}{2}(d_2+1)+b_2\right)\\&+(d_1'-1+a_0')\left(\frac{1}{2}(d_1+1)+b_1\right)
    +\left(d_1-1+a_0\right)\left(\frac{1}{2}(d_1'+1)+b_1'\right)\\&+(d_2'-1+a_0')\left(\frac{1}{2}(d_2+1)+b_2\right)+\left(d_2-1+a_0\right)\left(\frac{1}{2}(d_2'+1)+b_2'\right)
\end{align*}
We first deal with the terms without $b_i,b_i'$, denoted by $I_1$.
\begin{align*}
    I_1 &= \frac{1}{4}(d_1+1)(d_2'+1)+\frac{1}{4}(d_1'+1)(d_2+1)+\frac{1}{2}(d_1+1)(d_1'-1+a_0')+\frac{1}{2}(d_1-1+a_0)(d_1'+1)\\& +\frac{1}{2}(d_2+1)(d_2'-1+a_0')+\frac{1}{2}(d_2-1+a_0)(d_2'+1)
\end{align*}
Let $d=d_1+d_2$, $d'=d_1'+d_2'$, $b=b_1+b_2$, and $b'=b_1'+b_2'$, then we have $a_0 + b+d \le n_1$ and $a_0' + b'+d' \le n_2$.
Moreover, we have that
\begin{align*}
    I_1 &= \frac{1}{2}a_0'(d+2)+\frac{1}{2}a_0(d'+2)  + d d'- \frac{3}{4}(d_1d_2'+d_2d_1')-\frac{3}{2}+\frac{1}{4}(d+d')\\
    &\leq  \frac{1}{2}\left(n_2-b'-d'\right)(d+2)+\frac{1}{2}\left(n_1-b -d\right)(d'+2)+dd' + \frac{1}{4}(d+d') \\
    &\le  \frac{1}{2}(n_1-b)d' +\frac{1}{2}(n_2-b')d + n_1+n_2-b-b'\\
    &\le  \frac{1}{2}(n_1-b)(n_2-b') +\frac{1}{2}(n_2-b')(n_1-b) + n_1+n_2-b-b'\\
    &\le  n_1n_2 -(b-1)n_2 -(b'-1)n_1 + bb'
\end{align*}
Then we deal with the terms with $b_i$ or $b_i'$, denoted by $I_2$. 
\begin{equation*}
\begin{aligned}
    I_2=&b_1b_2'+\frac{1}{2}b_1(d_2'+1)+\frac{1}{2}b_2'(d_1+1)+b_1'b_2+\frac{1}{2}b_1'(d_2+1)+\frac{1}{2}b_2(d_1'+1)\\
    &+ b_1(d_1'-1+a_0')+b_1'(d_1-1+a_0)+b_2(d_2'-1+a_0')+b_2'(d_2-1+a_0)\\
    \leq & b_1b_2'+b_1'b_2+b_1\left(\frac{1}{2}(d_2'+1)+a_0'+d_1'-1\right)+b_1'\left(\frac{1}{2}(d_2+1)+a_0+d_1-1\right)\\&+b_2\left(\frac{1}{2}(d_1'+1)+a_0'+d_2'-1\right)+b_2'\left(\frac{1}{2}(d_1+1)+a_0+d_2-1\right)\\
    \leq &b_1b_2'+b_1'b_2+b(n_2-b')+b'(n_1-b)\\
    \leq & bn_2+b'n_1-2bb'+b_1b_2'+b_1'b_2.
\end{aligned}
\end{equation*}
Summing up $I_1$ and $I_2$, we obtain 
\begin{align*}
    \phi(xy)&\leq n_1n_2 -(b-1)n_2 -(b'-1)n_1 + bb' + bn_2 + b'n_1 -2bb' + b_1b_2'+b_1'b_2\\
    & \le n_1n_2 + n_1 + n_2 -bb' + b_1b_2' + b_1'b_2 \\ 
    & \le n_1n_2 + n_1+n_2
\end{align*}


The number of induced $P_4$'s contained in $\mathcal{O}$ is at most $f(n_1)\leq \frac{1}{2}n_1^2+C n_1\log n_1$ by induction, and the number of induced $P_4$'s contained in $\mathcal{O}'$ is at most $f(n_2)\leq \frac{1}{2}n_2^2+C n_2\log n_2$ by induction.
Therefore, the total number of induced $P_4$'s in $G$ is at most 

\begin{equation*}    
\begin{aligned} 
    &f(n_1)+f(n_2)+\phi(xy)\\ 
    \leq &\frac{1}{2}n_1^2+\frac{1}{2}n_2^2+n_1n_2+C n_1\log n_1+C n_2\log n_2+ n_1 + n_2\\ 
    =& \frac{1}{2}n^2+C n\log n+J, 
\end{aligned}
\end{equation*}
where $$J=Cn_1\log n_1+Cn_2\log n_2-Cn \log n + 3n+4.$$
Then, it is sufficient to prove $J\leq 0$, which is equivalent to
\begin{equation*}
    3n+4 \le Cn_1\log (n/n_1)+Cn_2\log (n/n_2) + 2C\log n.
\end{equation*}
Since $n \ge 400$, $n_1+n_2=n+2$, and $n_i\geq \frac{n}{200}, i=1,2$, we have $n_i \le \frac{99}{100}n$, $i=1,2$.
Then $\log(n/n_i) \ge \log (100/99)$, and the right-hand side is at least $\log (100/99) \cdot C n$.
When $N_0$ and $C$ are large enough, the above inequality holds.
\hfill $\square$ \par

\section{Count induced \texorpdfstring{$P_{k+1}$}{P\_\{k+1\}}'s in outerplanar graphs}\label{sec: induce Pk outerplanar}
In this section, we give the proof of Theorem \ref{thm: Pk outerplanar}.

\noindent
\textbf{Lower bound construction.}

\begin{figure}
    \centering
    \includegraphics[width=0.7\textwidth]{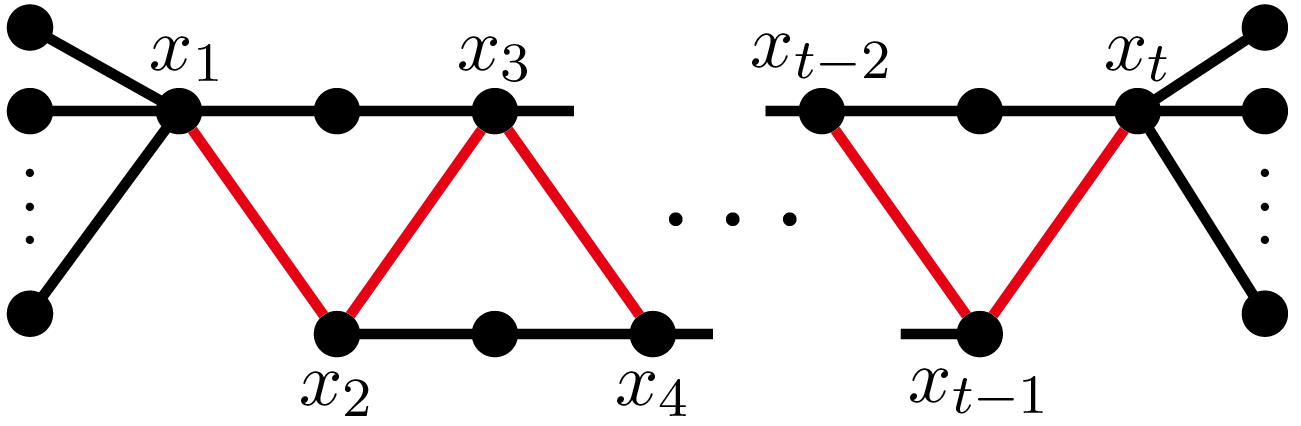}
    \caption{The graph $G_t'$. The red edges are from the path $P_t$.}\label{fig: lower bound construction}
\end{figure}

Let $P_t$ be a path on $t$ vertices $x_1,x_2,\ldots,x_t$, then define a graph $G_t$ by for each pair of vertices with distance two in $P_t$, connect them with a path of length two and the new vertex in the path is new for different pairs.
Let $G_t'$ be the graph obtained from $G_t$ by adding $\lfloor (n-2t+2)/2 \rfloor$ neighbors to $x_1$ and $x_t$ respectively~(see Figure~\ref{fig: lower bound construction}).
Then $G_{k-1}'$ is the lower bound construction for induced $P_{k+1}$ in outerplanar graphs.
Let us count the number of induced $P_{k+1}^{ind}$ in $G_{k-1}'$.

Let $h(t)$ be the number of induced $P_{t}^{ind}$ in $G_t$ starting at $x_1$.
When $t$ is small, we have $h(2) = 1$ and $h(3) = 2$.
Note that if the second vertex of the path is $x_2$, then the rest of the path corresponds to a $P_{t-1}^{ind}$ in $G_{t-1}$ starting at $x_2$.
If the second vertex of the path is not $x_3$, then it must be the new vertex connecting $x_1$ and $x_3$, and then the rest of the path corresponds to a $P_{t-2}^{ind}$ in $G_{t-2}$ starting at $x_3$.
Thus we have the recurrence relation
\begin{equation*}
    h(t) = h(t-1) + h(t-2).
\end{equation*}

Combining the initial values, we have that $h(t) = fib(t)$.
Then the number of induced $P_{k+1}^{ind}$ in $G_{k-1}'$ is at least $\frac{{(n-2k+3)}^2}{4}h(k-1) = \frac{{(n-2k+3)}^2}{4} fib(k-1)$.

\noindent
\textbf{Upper bound.}
Let $G$ be an outerplanar graph on $n$ vertices with maximum number of induced $P_{k+1}$'s.
As in the proof of Theorem~\ref{thm: P4 inducibility}, let $G'\supseteq G$ be the maximal outerplanar graph containing $G$ as a subgraph.
Then $G'$ admits a planar drawing with an outer cycle $\mathcal{C}$, and we draw the graph $G$ in the same way as $G'$ by removing the edges not in $G$.
Let $g(k)$ be the maximum number of induced $P_{k+1}$'s connecting two fixed vertices as end vertices of $G$.
Then clearly, we have that $p_{k+1}^{\ind}(G) \le g(k) \binom{n}{2}$.
It remains to estimate $g(k)$.
By outerplanarity and definition, we have $g(1) \le 1$ and $g(2) \le 2$.
We claim that $g(k) \le g(k-1) + g(k-2)$ for $k \ge 3$.

\begin{figure}
    \centering
    \begin{minipage}{0.4\textwidth}
    \centering
    \includegraphics[width=\textwidth]{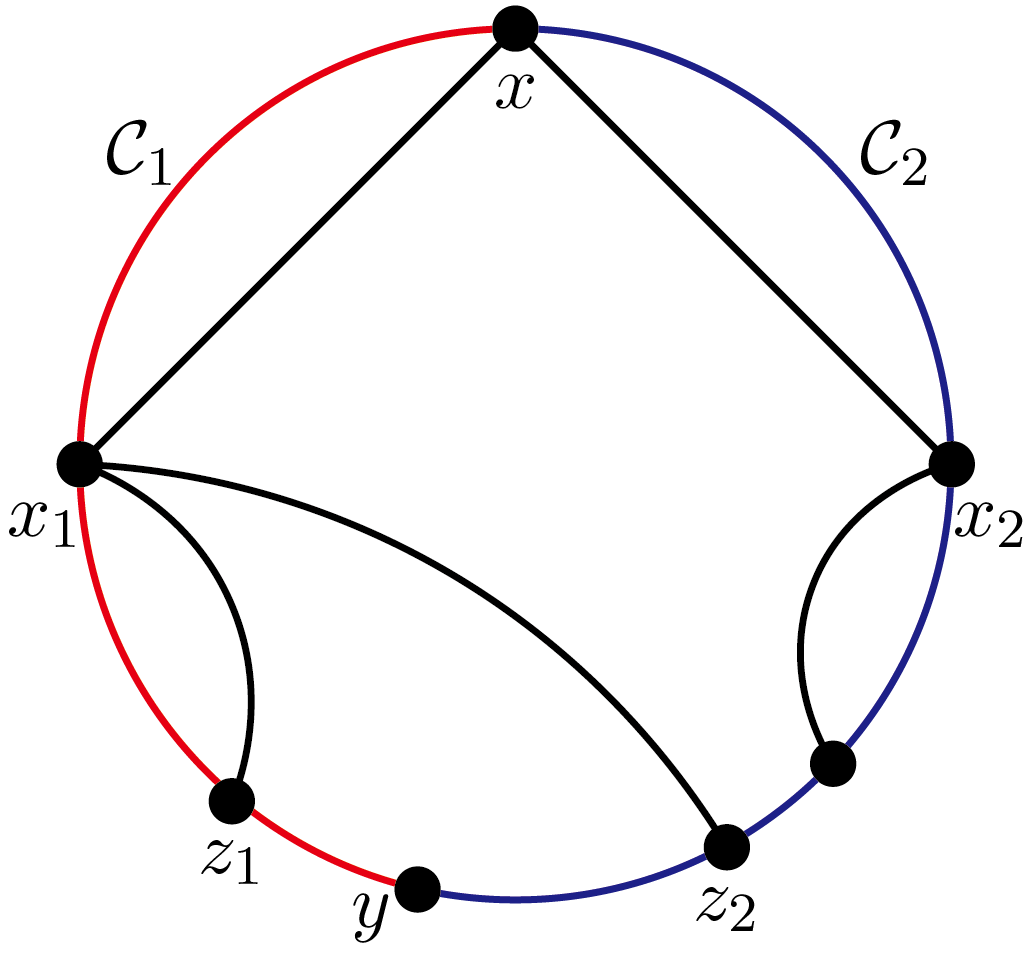}
    \caption{An illustration in the proof of the upper bound of Theorem~\ref{thm: Pk outerplanar}. 
    }\label{fig: upper bound fibonacci}
    \end{minipage}
\end{figure}

Let $x,y$ be two fixed vertices, and denote by $p^{\ind}_{k+1}(x,y)$ the number of induced $P_{k+1}$'s containing $x,y$ as end vertices.
It is sufficient to prove that $p^{\ind}_{k+1}(x,y) \le g(k-1) + g(k-2)$.
The vertices $x,y$ subdivide $\mathcal{C}$ into two arcs, denoted by $\mathcal{C}_1$ and $\mathcal{C}_2$~(see Figure~\ref{fig: upper bound fibonacci}).
Let $x_1$ be the closest neighbor of $x$ to $y$ in $\mathcal{C}_1$, that is, there is no neighbor of $x$ in $\mathcal{C}_1$ between $x_1$ and $y$.
Similarly, let $x_2$ be the closest neighbor of $x$ to $y$ in $\mathcal{C}_2$.
First, we consider the case when both $x_1$ and $x_2$ exist.
Let $P$ be an induced $P_{k+1}$ starting at $x$ and ending at $y$.
We claim that the neighbor of $x$ in $P$ is either $x_1$ or $x_2$.
Suppose otherwise, if the neighbor of $x$ in $P$, say $z$, is between $x$ and $x_1$ in $\mathcal{C}_1$, then there is a path from $z$ to $y$ without using $x_1$ since the path is induced.
It would contradict the outerplanarity.
Similarly, if $z$ is between $x$ and $x_2$ in $\mathcal{C}_2$, we can get a contradiction.
If $x_1$ or $x_2$ does not exist, then we have only one choice for the neighbor of $x$ in $P$, which implies that the number of induced $P_{k+1}$'s containing $x,y$ as end vertices is at most $p^{\ind}_{k}(x_1, y)$ or $p^{\ind}_{k}(x_2, y)$, which is at most $g(k-1)$ in both cases.
From now on, we may assume $x_1$ and $x_2$ both exist. Then $p^{\ind}_{k+1}(x,y) \le p^{\ind}_{k}(x_1, y) + p^{\ind}_{k}(x_2, y)$.

When the neighbor of $x$ in $P$ is $x_1$, then similarly, there are at most two choices for the neighbor of $x_1$ in $P$, say $z_1$ and $z_2$.
In this case, $z_1$ and $z_2$ are the closest neighbors of $x_1$ to $y$ in two arcs divided by $x_1$ and $y$ respectively.
Moreover, we may assume $z_1$ is on the arc with end points $x_1,y$ without $x$, and $z_2$ is on the arc with end points $x_2,y$ without $x$~(see Figure~\ref{fig: upper bound fibonacci}).
If either of $z_1,z_2$ does not exist, then we have that $p^{\ind}_{k}(x_1, y) \le g(k-2)$, and thus, $p^{\ind}_{k+1}(x,y) \le g(k-1) + g(k-2)$.
So we may assume $z_1$ and $z_2$ both exist. 
In this case, $x_2$ has no neighbor on the arc with end points $x_2$ and $y$ containing $x$.
Therefore, $p^{\ind}_{k-1}(x_2, y) \le g(k-2)$ and thus $p^{\ind}_{k+1}(x,y) \le g(k-1) + g(k-2)$.

Since $g(k)$ satisfies the Fibonacci recurrence relation with $g(1) \le 1$ and $g(2) \le 2$, we have that $g(k) \le fib(k+1)$.
\bibliography{ref.bib}
\bibliographystyle{wyc4}

\end{document}